\documentclass[12pt, reqno, a4paper]{amsart}

\usepackage{hyperref, amssymb, amsmath, enumerate, amsfonts, amsthm, mathrsfs, url, bm}

\usepackage{color}

\usepackage[margin=1.5in]{geometry}

\RequirePackage{mathrsfs} \let\mathcal\mathscr
\numberwithin{equation}{section}

\usepackage{mystyle}

\makeatletter
\let\@@pmod\pmod
\DeclareRobustCommand{\pmod}{\@ifstar\@pmods\@@pmod}
\def\@pmods#1{\mkern4mu({\operator@font mod}\mkern 6mu#1)}
\makeatother

\newcommand{\ZZ}{\mathbb{Z}}
\newcommand{\TT}{\mathbb{T}}
\newcommand{\CC}{\mathbb{C}}
\newcommand{\cR}{\mathcal{R}}
\renewcommand{\leq}{\leqslant}
\renewcommand{\geq}{\geqslant}

\begin{document}

\title[A Roth-type theorem in the squares]{A transference approach to a Roth-type theorem in the squares}

\author{T.~D. Browning}
\address{School of Mathematics\\
University of Bristol\\ Bristol\\ BS8 1TW}
\email{t.d.browning@bristol.ac.uk}
\author{S.~M. Prendiville}
\address{Mathematical Institute\\
University of Oxford\\ Oxford\\ OX2 6GG}
\email{sean.prendiville@maths.ox.ac.uk}


\date{\today}

\begin{abstract}
We show that any subset of the squares of positive relative upper density contains non-trivial solutions to a translation-invariant 
linear equation in five or more variables, with explicit quantitative bounds.  As a consequence, we establish the partition regularity of any diagonal quadric in five or more variables whose coefficients sum to zero. 
Unlike previous approaches, which are limited to equations in seven or more variables, we employ  transference technology of Green
to import  bounds from the linear setting.  
\end{abstract}

\maketitle

\setcounter{tocdepth}{1}
\tableofcontents
\thispagestyle{empty}

\section{Introduction}

Initiated by Roth \cite{roth53, roth53II}, there has been much work (see \cite{bloom12, schoensisask}) on bounding the density of subsets of integers lacking solutions to a single linear equation
\begin{equation}\label{MainEqn}
c_1y_1 + \dots + c_s y_s = 0 ,
\end{equation}
with integer coefficients.
It is customary to impose the condition that $c_1 + \dots + c_s = 0$ in order to avoid the existence of a congruence class lacking solutions.  Our knowledge is more formative for the analogous non-linear equation
\begin{equation}\label{bustard}
c_1x_1^2 + \dots + c_s x_s^2 = 0 ,
\end{equation}
where the aim is to show that sufficiently dense 
subsets of integers have generic solutions. 
Given a union $K$ of $k$ proper subspaces of the hyperplane \eqref{MainEqn}, 
our notion of 
  non-generic solutions
is captured by $K$-triviality,   
  where
a solution $(x_1,\dots,x_s)$ is said to be  {\em $K$-trivial} if  $(x_1^2, \dots, x_s^2) \in K$.
For instance, one could regard solutions with all $x_i$ equal as $K$-trivial, by taking $K$ to be the diagonal subspace
\begin{equation}\label{diagonal}
K := \set{(y, \dots, y) : y \in \Q}, 
\end{equation}
for which $k = 1$.   More generally, taking
$$
 K := \bigcup_{i \neq j} \set{\vy \in \Q^s : y_i = y_j \text{ and } \vc \cdot \vy = 0}, 
$$
we have $k = \binom{s}{2}$ and the $K$-nontrivial solutions correspond to those in which all $x_i$ are distinct.

All work on  equation \eqref{bustard} (see \cite{smith, 
keil14, keilpreprint, henriot}) is based on an adaptation of the density increment approach to Roth's theorem, and is at present limited to equations in at least $s \geq 7$ variables and with at least two positive and two negative coefficients.  
Adapting the transference technology of Green \cite{greenprimes}, we offer an alternative approach to this problem, an approach which is applicable to a wider class of equations, in particular those in  $s \geq 5$ variables and with no sign restriction on the coefficients.

\begin{theorem}\label{thm:main} 
Let  $c_1, \dots, c_s  \in \Z\setminus\set{0}$ with $c_1+\dots + c_s = 0$ and $s \geq 5$.  Let $K$ denote a union of at most $k$ proper subspaces of the hyperplane \eqref{MainEqn} each containing the 
diagonal subspace \eqref{diagonal}. 
If $A \subset [X]$ is such that the only solutions to 
\begin{equation}\label{eqn:heron}
c_1x_1^2 + \dots + c_s x_s^2 = 0 \qquad (x_i \in A), 
\end{equation} 
are $K$-trivial, then for any $\eps > 0$ we have
$$
|A| \ll_{\vc, k, \eps}  \frac{X}{(\log\log\log X)^{\frac{s}{2} -1 -\eps}}.
$$
\end{theorem}

Let  $c_1, \dots, c_s$ be as in the statement of Theorem \ref{thm:main}.
Then this result implies that any  subset $A\subset \N$,
with  positive upper density 
\begin{equation}\label{upperdensitydef}
\limsup_{X \to \infty} \frac{|A\cap [X]|}{X}>0,
\end{equation}
must contain generic 
solutions to \eqref{eqn:heron}.

A Diophantine equation $f(x_1, \dots, x_s)  = 0 $ is called \emph{partition regular} if, given any partition of $\N$ into finitely many colour classes, at least one class contains a solution to the equation in which all $x_i$ are distinct.  An old problem of Erd\H{o}s and Graham \cite{erdosgraham} asks if the Pythagorean equation $x^2 + y^2 = z^2$ is partition regular.  Graham \cite{graham} has offered \$250 for a resolution of this conjecture.  Similarly, it is unknown if the equation $x^2 + y^2 = 2z^2$ is partition regular (see \cite{frantzikinakishost}).  Given an extra two variables, we can answer questions of the latter type.

\begin{corollary}
Let  $c_1, \dots, c_s  \in \Z\setminus\set{0}$ with $c_1+\dots + c_s = 0$ and $s \geq 5$.
Then the equation 
$
c_1 x_1^2 + \dots + c_sx_s^2 = 0
$ is partition regular.  
\end{corollary}
This follows from Theorem \ref{thm:main}, 
since at least one colour class must have positive upper density \eqref{upperdensitydef}.  

Define the \emph{Rado number} (see \cite[p.103]{ramseytheory}) of the equation \eqref{eqn:heron} to be the smallest positive integer $R_{\vc}(r)$ such that any $r$-colouring of the interval $\set{1, 2, \dots, R_{\vc}(r)}$ results in at least one monochromatic tuple $(x_1, \dots, x_s)$ satisfying \eqref{eqn:heron} with all $x_i$ distinct.  Since any $r$-colouring of $[X]$ results in a colour class of size at least $X/r$, one can use the estimate given in Theorem \ref{thm:main} to deduce the following result.
\begin{corollary}
Let  $c_1, \dots, c_s  \in \Z\setminus\set{0}$ with $c_1+\dots + c_s = 0$ and $s \geq 5$.  Then there exists a constant $C_{\vc}>0$, depending only on $c_1,\dots,c_s$, such that
$$
R_{\vc}(r) \leq \exp\exp\exp(C_{\vc} r)
$$
\end{corollary}

The study of density bounds for sets lacking solutions to \eqref{eqn:heron} was initiated by Smith \cite{smith} who proceeded by considering instead  the system of equations
\begin{equation}\label{system}
\begin{split}
c_1x_1^2 + \dots + c_s x_s^2 & = 0, \\
c_1x_1+ \dots + c_s x_s & = 0 .
\end{split}
\end{equation}
The advantage of this system is that it is translation-invariant, in that the set of solutions remains unchanged under a shift from the diagonal subspace \eqref{diagonal}.  This allows one to follow Roth's original approach to the equation \eqref{MainEqn}, using an $L^p$-estimate (mean value) to deduce that if a set $A$ lacks solutions to \eqref{system}, then there is an exponential sum associated to $A$ which is large.  It follows that $A$ has a density increment on a long progression of the form $a + q \cdot [N']$, and so (by translation-invariance) there exists a denser set $A' \subset [N']$ lacking solutions to \eqref{system}.  By iterating this argument, a bound can eventually be extracted on the density of the original set $A$. 

Smith \cite{smith} accomplished this programme for equations in 9 or more variables, obtaining a bound of the form $|A| \ll X(\log\log X)^{-c}$ for some small $c >0$.  This small exponent resulted from a use of quadratic Fourier analysis, a tool that Keil \cite{keil14} avoided, obtaining a bound of the form $X(\log\log X)^{-1/15}$.  Moreover, Keil incorporated a restriction $L^p$-estimate in order to handle equations in seven or more variables.
Subsequently, Henriot \cite{henriot} has improved the cardinality bound to $X(\log X)^{-c}$ for 
$s\geq 7$ and an exponent  $c > 0$
that depends on the choice of coefficients $c_1,\dots,c_s$.

In all these approaches, one requires that there are at least two positive and two negative coefficients in order to guarantee that the system \eqref{system} has a non-diagonal real solution.  Furthermore, $s \geq 7$ is a natural limit of the analytic method, as for $s \leq 6$ the system may possess more solutions than that predicted by the usual probabilistic heuristic.  However, if one is only interested in sets of integers lacking solutions to the top equation \eqref{eqn:heron}, then neither of these conditions should be necessary.

Analysing the single equation \eqref{eqn:heron} removes translation-invariance, making a density increment approach more problematic.  Instead we implement the transference principle devised in Green \cite{greenprimes} to import results from the dense linear setting to the sparse linear setting, viewing a dense subset of the squares as a sparse subset of the integers.  Although Roth's theorem only holds for dense subsets of integers, we show that an appropriately normalised indicator function of the squares has sufficient pseudorandomness properties that any dense subset of squares can be modelled by a dense subset of integers.  Applying Roth's theorem in the dense setting, we are then able to transfer back to the sparse setting to bound the size of our set of squares.  We describe this method further in \S 
\ref{s:2}.

It is worth remarking that Gowers' quantitative version of Szemer\'edi's theorem \cite{gowers01} yields bounds for sets lacking solutions to very many homogeneous Diophantine equations, an observation the second author learnt from Trevor Wooley.  
(A variant of this idea was put to use by Br\"udern et al \cite{bruedernetal}.)  For the sake of completeness, we include a proof of the folklore argument.
\begin{proposition}\label{SZcheat}
Let $F \in \Z[x_1, \dots, x_s]$ be a homogeneous polynomial with integer coefficients and let $K$ denote a finite union of proper subspaces of $\Q^s$, one of which is equal to the diagonal subspace \eqref{diagonal}.  Suppose that the solution set $F(x_1, \dots, x_s) = 0$ contains a two-dimensional rational subspace which contains the diagonal \eqref{diagonal} and which is not contained in $K$.  Then any set $A\subset [X]$ which satisfies the implication
$$
F(\vx) = 0 \quad (\vx \in A^s) \qquad \implies \qquad \vx \in K,
$$
has cardinality
$$
|A| \ll \frac{X}{(\log \log X)^c},
$$
for some positive constant $c = c(F) > 0$.
\end{proposition}

\begin{proof}
Let $V$ denote a two-dimensional rational subspace of the solution set $F(\vx) = 0$, which contains the diagonal \eqref{diagonal} and is not contained in $K$.  Since $(1, \dots, 1) \in V\cap K$ and $\dim V = 2$, there exists $\vx \in V\setminus K$ which is linearly independent of $(1, \dots, 1)$.  Clearing denominators and shifting along the diagonal subspace by a sufficiently large integer, we may assume that $\vx \in ([0, k)\cap \Z)^s$ for some $k \in \N$.  

Notice that for any $a, q \in \Z$ the tuple $\mathbf{y}=(a + q x_1, \dots, a+qx_s)$ lies in $V$, so that $F(\mathbf{y}) = 0.$  Suppose now that  $A$ contains an arithmetic progression $a, a+q ,\dots, a+ (k-1)q$ of length $k$. Then it follows that  $A^s$ contains $\mathbf{y}$, which is also a solution to $F = 0$.  Moreover, this solution is not contained in $K$, since otherwise 
$$
\vx = \frac{\mathbf{y} - (a, \dots, a)}{q} \in K.
$$  
Thus $A$ cannot contain an arithmetic progression of length $k$, so that 
$
|A| \ll X(\log \log X)^{-c_k},
$
for some $c_k > 0$, 
by Gowers's bounds in Szemer\'edi's theorem \cite{gowers01}.
\end{proof}

This argument covers those equations of the form \eqref{eqn:heron} tackled by Smith, Keil and Henriot, namely those in which 
$c_1 + \dots + c_s = 0$ 
for $s \geq 7$,
and 
where
there are at least two positive and two negative coefficients.  
Indeed, an application of the Hardy--Littlewood method (as carried out by Keil \cite{keil14}) 
gives  an integer vector $\vx \in \Z^s\setminus K$ satisfying the system \eqref{system}
under these conditions.   Expanding the square, one can check that for any $\lambda, \mu \in \Q$ the vector $(\lambda + \mu x_1, \dots, \lambda + \mu x_s)$ also satisfies \eqref{system}.  In particular, the solution set of \eqref{eqn:heron} contains the two dimensional subspace generated by the diagonal and $\vx$, a subspace which is not contained in $K$.

Unlike the bounds obtained via Proposition \ref{SZcheat}, the advantage of the bounds of Smith and Keil is that their logarithmic exponent is \emph{uniform} in the choice of coefficients, whilst Henriot succeeds in replacing the double logarithm derived from Gowers's bounds by a single logarithm.

Note that one 
 cannot use Proposition \ref{SZcheat} to treat all equations covered by Theorem \ref{thm:main}.  For instance, the convex equation
$$
x_1^2 + x_2^2 + x_3^2 + x_{4}^2 = 4x_5^2
$$
is of the form \eqref{eqn:heron} but is not satisfied by a two dimensional subspace containing the diagonal.  Should the solution set of this equation contain such a set, then this subspace can be written in the form
\begin{equation}\label{twodim}
\set{(\lambda, \dots, \lambda) + \mu \vx : \lambda,\mu \in \Q}.
\end{equation}
Expanding the squares in the equation $\sum_i (\lambda + \mu x_i)^2 = 4(\lambda + \mu x_5)^2$ reveals that $\vx$ also satisfies the system \eqref{system}.  Substituting the linear equation into the quadratic equation shows that 
$$
x_1^2 + x_2^2 + x_3^2+ x_{4}^2 = 4(x_1 +x_2+ x_3+ x_{4})^2, 
$$
which forces $\vx$ to be an element of the diagonal, contradicting the fact that \eqref{twodim} is two dimensional.

Finally we remark that in all likelihood our methods can be adapted to handle diagonal equations of degree $k$ in $s$ variables, provided that one has a restriction estimate at exponent $p<s$ for the appropriate exponential sum.  This gives an alternative approach to recent results of Henriot \cite{henriothigherpowers}, albeit with much weaker density bounds.  Whereas Henriot's methods rely on a restriction estimate associated to the Vinogradov system
\begin{align*}
x_1^k + \dots + x_s^k & = y_1^k + \dots + y_s^k\\
&\hspace{0.2cm}\vdots\\
x_1 + \dots + x_s & = y_1 + \dots + y_s,
\end{align*}
 the transference approach of this paper would merely  require a mean value estimate for the top equation in the system.  For small values of $k$, it appears that the transference approach may  allow one to improve the number of variables required in Henriot's work.  
 For instance, an eighth moment estimate of Vaughan \cite{vaughancubes} should yield the partition regularity of a diagonal  cubic equation in 9 or more variables, whereas the work of Henriot \cite{henriothigherpowers} requires at least 13 
 variables when specified to this situation.

\subsection*{Acknowledgements} 
We are very grateful to Thomas Bloom, Sam Chow, Surya Ramana and Trevor Wooley for a number of useful comments.
Special thanks are due to the anonymous referees for several helpful suggestions that have improved the exposition of the paper.
Whilst working on this paper the authors were 
supported by ERC grant \texttt{306457}.  

\section{The structure of our argument}\label{s:2}

\begin{definition}[Majorant]
A  \emph{majorant on $[N]$} is a non-negative function $\nu : \Z \to [0, \infty)$ with support contained in the interval $[N]$.
\end{definition}
The next definition, and much of this section, follows the exposition of Tao \cite[\S 1.7]{taohigherorder}.

\begin{definition}[$\vc$-pseudorandom]
Let $\vc=(c_1, \dots, c_s) \in\Z^s$. We say that a majorant $\nu$ is \emph{$\vc$-pseudorandom} if for any $\delta > 0$ there exists $c(\delta) > 0$ such that for any $0 \leq f \leq \nu$ we have the implication
$$
 \sum_x f(x) \geq \delta \sum_x \nu(x) \quad \implies \quad \sum_{\vc\cdot\vx = 0} \prod_{i=1}^s f(x_i) \geq c(\delta) \sum_{\vc\cdot \vx = 0}\prod_{i=1}^s \nu(x_i).
$$
\end{definition}

The indicator function of an interval provides an important example of a $\vc$-pseudorandom majorant, 
as can be seen from the following result of Bloom \cite{bloom12} (building on work of Sanders \cite{sanders}).

\newtheorem*{bloom}{Bloom's theorem}
\begin{bloom}[\cite{bloom12}]
Let $c_1,\dots,c_s\in \Z \setminus\set{0}$ with $s \geq 3$ and $c_1 + \dots + c_s = 0$. Then the indicator function $1_{[N]}$ is $\vc$-pseudorandom with quantitative dependence 
$$
c(\delta) \gg_{\vc} \exp\brac{-C/\delta^{\frac{1}{s-2-\eps}}}
$$
for some absolute constant $C = C(s, \eps)$ and any $\eps > 0$.
\end{bloom}

Our aim is to show that there exists a function related to the indicator function of the squares which is also $\vc$-pseudorandom, at least when $s \geq 5$.  There are three conditions which, taken together, are sufficient to guarantee that a majorant is $\vc$-pseudorandom. We discuss these in the remainder of this section.

Define the $L^p$-norm of a function $f : \Z \to \C$ with respect to counting measure, so for example
$$
\norm{f}_1 := \sum_n |f(n)|.
$$
Provided that this $L^1$-norm is finite, we define the Fourier transform of $f$ at $\alpha \in \T := \R/\Z$ by 
\begin{equation}\label{FourierDef}
\hat{f}(\alpha) := \sum_n f(n) e(\alpha n).
\end{equation}
For functions on $\T$, all $L^p$-norms are taken with respect to the Haar probability measure, so that
$
\bignorm{\hat{f}}_2 = \norm{f}_2.
$

\begin{definition}[Fourier decay] We say that a majorant $\nu$ has \emph{Fourier decay of level $\theta$} if
$$
\norm{\hat{\nu} - \hat{1}_{[N]}}_\infty \leq \theta N.
$$
\end{definition}

The relevance of this definition is demonstrated in work of Green \cite{greenprimes}.  The idea is that one can approximate an unbounded function by a bounded function, provided the unbounded function has a majorant with sufficient Fourier decay.  For a proof of the following form of the bounded approximation lemma we refer the reader to the second author's survey \cite{boundedapproximation}.

\begin{lemma}[Bounded approximation]
Suppose that the majorant $\nu$ has Fourier decay of level $\theta$.  Then for any $0 \leq f \leq \nu$ there exists a bounded function $0 \leq g \leq 1_{[N]}$ such that
$$
\bignorm{\hat{f} - \hat{g}}_{\infty} \ll \frac{N}{\log(1/\theta)^{3/2}}.
$$
\end{lemma}

The advantage of a bounded approximant is that it is amenable to an application of Bloom's theorem, and should therefore count many solutions to \eqref{MainEqn}.  We wish to approximate the characteristic function of a set of squares weighted by an unbounded majorant.  Given a sufficiently good approximation, we might therefore hope that the number of solutions to \eqref{MainEqn} counted by our unbounded function is similar to the count for the bounded function.  This is indeed the case, provided that our majorant possesses an additional property.

\begin{definition}[Restriction at $p$] We say that a majorant $\nu$  on $[N]$ satisfies a \emph{restriction estimate at exponent $p$} if 
$$
\sup_{|\phi| \leq \nu} \int_{\T} \abs{\hat{\phi}(\alpha)}^p \intd \alpha \ll_p \norm{\nu}_1^p N^{-1}.
$$
\end{definition}

The utility of this property for our purposes is demonstrated in the following lemma, which we prove in \S \ref{sec:configcontrol}.

\begin{lemma}[Restriction implies configuration control]\label{lem:configcontrol}  Suppose that $\nu$ is a majorant with $\norm{\nu}_1 \ll N$ and which satisfies a restriction estimate at some exponent $2 \leq p < s$.  
Let $c_1, \dots, c_s \in \Z\setminus\set{0}$.  
Then for functions $f, g : \Z \to \C$ with $|f| \leq \nu$ and $|g| \leq 1_{[N]}$ we have
\begin{equation*}
\abs{\sum_{\vc\cdot\vx = 0} \brac{\prod_{i = 1}^s f(x_i) -  \prod_{i = 1}^s g(x_i)}} \ll_{p,s} N^{s-1} \brac{ N^{-1} \bignorm{\hat{f}-\hat{g}}_{\infty}}^{1-\set{p}}.
\end{equation*}
\end{lemma}

Finally, 
we
 need to ensure that our majorant does not only count $K$-trivial solutions to \eqref{MainEqn}.

\begin{definition}[Saves $\eta$ on the $K$-trivial solutions]\label{saves eta} 
Let $\eta>0$ and let 
$K$ denote a finite union of $k$ proper subspaces of the hyperplane \eqref{MainEqn}.  We say that the majorant $\nu$ \emph{saves $\eta$ on the $K$-trivial solutions} if
$$
\sum_{\vx \in K} \prod_{i=1}^s \nu(x_i) \ll_{s, k} \norm{\nu}_1^sN^{-1-\eta}.
$$
\end{definition}

With these definitions in hand we are able to prove the following.

\begin{proposition}[Sufficient majorant conditions]\label{BourbakiProp}  Let $c_i \in \Z\setminus\set{0}$ with $c_1 + \dots + c_s = 0$ and $s\geq 3$.  Let $K$ denote a finite union of $k$ proper subspaces of the hyperplane \eqref{MainEqn}. Suppose that $\nu$ is a majorant satisfying 
\begin{itemize}
\item the restriction estimate at some exponent $p\in [s-1, s)$;
\item  Fourier decay of level $\theta \leq 1$;
\item  the $K$-trivial estimate with saving $\eta$.
\end{itemize}  Then any set $A \subset \supp(\nu)$ containing only $K$-trivial solutions to \eqref{MainEqn} satisfies
\begin{equation}\label{BourbakiBound}
\sum_n 1_A(n) \nu(n) \ll_{\vc,p,k,\eta,\eps} \frac{N}{\min\set{
\log\log(1/\theta),\ \log N}^{s-2-\eps}} .
\end{equation}
\end{proposition}

Note that in applications we do not expect to be able to take $\theta$ any smaller that $N^{-c}$, in which case the minimum occurring in  the right hand side of \eqref{BourbakiBound} is $\log\log(1/\theta)$.

The statistic
$$
\sum_n 1_A(n) \nu(n)
$$
is not necessarily the same as the quantity $|A|$, to which Theorem \ref{thm:main} pertains.  However, in most applications there is simple polynomial dependence between these quantities.  We discuss this in the context of our application in
Lemma \ref{lem:MandarinDuck}.

\begin{proof}[Proof of Proposition \ref{BourbakiProp}]
Define $f := 1_A \nu$ and set
$$
\delta := N^{-1}\sum_n f(n).
$$  
We wish to determine a lower bound on $\delta$ in terms of $s, p, k, \eta, \vc$ and $\theta$ which guarantees that $A$ contains at least one $K$-nontrivial solution to \eqref{MainEqn}.
By the bounded approximation lemma, there exists a function $0 \leq g \leq 1_{[N]}$ such that 
$$
\bignorm{\hat{f} - \hat{g}}_{\infty} \ll \frac{N}{\log(1/\theta)^{3/2}}.
$$
Since the $L^1$-norm of a non-negative function is equal to the sup-norm of its Fourier transform, we have
\begin{align*}
\sum_n g(n) = \norm{\hat{g}}_\infty  \geq \bignorm{\hat{f}}_\infty - \bignorm{\hat{f} - \hat{g}}_\infty 
&= \sum_n f(n) - \bignorm{\hat{f} - \hat{g}}_\infty \\
&= \delta N + O\brac{\frac{N}{\log(1/\theta)^{3/2}}}.
\end{align*}
Thus we may deduce that 
$$
\sum_n g(n) \geq (\delta/2) N,
$$
provided we assume 
\begin{equation}\label{delta1}
\delta \geq \frac{C}{\log(1/\theta)^{3/2}},
\end{equation}
 for some absolute constant $C$.  Bloom's theorem then implies that  there exists $c(\delta/2) > 0$ such that 
$$
\sum_{\vc \cdot \vx = 0} \prod_{i=1}^s g(x_i) \gg_{\vc} c(\delta/2) N^{s-1}.
$$

Since our majorant has Fourier decay of level $\leq 1$, we have $\|\nu\|_1 = \|\hat{\nu}\|_\infty \leq \|\hat{\nu}-\hat{1}_{[N]}\|_\infty + \|\hat{1}_{[N]}\|_\infty \ll N$.  Thus Lemma \ref{lem:configcontrol} yields
$$
\sum_{\vc \cdot \vx = 0} \prod_{i=1}^s f(x_i) = \sum_{\vc \cdot \vx = 0} \prod_{i=1}^s g(x_i) + O_{s, p}\brac{\frac{N^{s-1}}{\log(1/\theta)^{\frac{3(s-p)}{2}}}}.
$$
Hence, if we ensure that
\begin{equation}\label{delta2}
c(\delta/2) \geq \frac{C_{s,p, \vc}}{\log(1/\theta)^{\frac{3(s-p)}{2}}}
\end{equation}
for some absolute constant $C_{s,p, \vc}$, we may conclude that
$$
\sum_{\vc \cdot \vx = 0} \prod_{i=1}^s f(x_i) \gg_{\vc} c(\delta/2) N^{s-1}.
$$

Yet if $A$ contains only $K$-trivial solutions to \eqref{MainEqn}, then the fact that $\nu$ saves $\eta$ on the $K$-trivial solutions gives 
$$
\sum_{\vc \cdot \vx = 0} \prod_{i=1}^s f(x_i) \leq \sum_{\vx \in K}  \prod_{i=1}^s \nu(x_i) \ll_{s,k} N^{s-1 - \eta},
$$
again using $\norm{\nu}_1 \ll N$. 
Hence $A$ must contain a $K$-nontrivial solution on ensuring that 
\begin{equation}\label{delta3}
c(\delta/2) \geq C_{s,k, \vc} N^{-\eta}
\end{equation}
for some absolute constant $C_{s,k, \vc}$.

Employing Bloom's lower bound for $c(\delta/2)$, our conditions \eqref{delta1}, \eqref{delta2} and \eqref{delta3} reduce to ensuring that for any $\eps >0$ there exists an absolute constant $C_{s, p, k, \eta,\vc, \eps}$ such that 
$$
\delta \geq C_{s, p, k, \eta,\vc, \eps} \max\set{ \frac{1}{\log(1/\theta)^{3/2}}, \frac{1}{\bigbrac{\log\log(1/\theta)}^{s-2-\eps} }, \frac{1}{(\log N)^{s-2-\eps}}}.
$$
The result is now obvious.
\end{proof}

 Heuristically, our strategy is to show that the following weighted indicator function of the squares is a suitable majorant:
\begin{equation}\label{IntroNu}
\nu(n) := \begin{cases} 2 \sqrt{n} & \text{if $n = x^2$ for some $x \in [X]$,}\\
					0 & \text{otherwise}.\end{cases}
\end{equation}
Taking $N := X^2$ we see that $\nu$ is supported on the interval $[N]$ with $\norm{\nu}_{1} \sim N$.
It follows from work of Bourgain \cite{bourgain89} that this majorant satisfies a restriction estimate at any exponent $p > 4$.  Moreover, a standard point-counting argument (see \S \ref{sec:KTriv}) shows that this majorant saves at least $1/2$ (essentially) on the $K$-trivial solutions.  However, the majorant does not exhibit sufficient Fourier decay. Indeed,  it is well-known from the classical circle method that there are values of $\alpha$ for which
$
\abs{\hat{\nu}(\alpha) - \hat{1}_{[N]}(\alpha)} \gg N.
$ 
(For example, take $\alpha=a/q$ with $2 < q \ll 1$.)  These problematic values of $\alpha$ arise from the major arcs, and are a consequence of the fact that the squares are not equidistributed in congruence classes to small moduli.  

We  overcome this difficulty by  implementing the $W$-trick,  as found in Green \cite{greenprimes}.
Our $W$-tricked version of the majorant \eqref{IntroNu} is found in \S \ref{sec:Wtrick}. 
Although this allows us to verify the existence of non-trivial Fourier decay in 
\S \ref{sec:ExpSums},  the restriction estimate no longer follows directly from Bourgain's work.  We show how to modify his proof in \S \ref{sec:restriction}.  Finally, in \S \ref{sec:KTriv} we show that the $W$-tricked majorant saves $1/5$ on the $K$-trivial solutions and in \S \ref{sec:MainThmProof} we bring everything together to prove Theorem \ref{thm:main} via Proposition \ref{BourbakiProp}.

\section{Configuration control}\label{sec:configcontrol}

The aim of this section is to establish Lemma \ref{lem:configcontrol}.
Let $c_1, \dots, c_s \in \Z\setminus\set{0}$ and suppose that $\nu$ is a majorant with $\norm{\nu}_1 \ll N$ and which satisfies a restriction estimate at exponent $2 \leq p < s$.  Then we need to show that  for any functions $f, g : \Z \to \C$ with $|f| \leq \nu$ and $|g| \leq 1_{[N]}$ we have
\begin{equation*}
\abs{\sum_{\vc\cdot\vx = 0} \brac{\prod_{i = 1}^s f(x_i) -  \prod_{i = 1}^s g(x_i)}} \ll_{p,s} N^{s-1} \brac{ N^{-1} \bignorm{\hat{f}-\hat{g}}_{\infty}}^{1-\{p\}}.
\end{equation*}
On recording the telescoping identity
$$
\prod_{i = 1}^s f(x_i) -  \prod_{i = 1}^s g(x_i) = \sum_{j=1}^s \bigbrac{f(x_j) - g(x_j)}\prod_{i < j} f(x_i)\prod_{i>j} g(x_i),
$$
we see that  Lemma \ref{lem:configcontrol} is implied by the following result.

\begin{lemma}[Configuration control]\label{lem:buzzard}  Let $c_1, \dots, c_s \in \Z\setminus\set{0}$.  Suppose that $\nu$ is a majorant with $\norm{\nu}_1 \ll N$ and which satisfies a restriction estimate at some exponent $2 \leq p < s$.  Then for any functions $f_1, \dots, f_s : \Z \to \C$, each satisfying $|f_i| \leq \nu$ or $|f_i| \leq 1_{[N]}$, we have
\begin{equation}\label{configcontroleqn}
\abs{\sum_{\vc\cdot\vx = 0} \prod_{i = 1}^s f_i(x_i)} \ll_{p} N^{s-1}\brac{N^{-1}\min_i\bignorm{\hat{f}_i}_{\infty}}^{1-\set{p}}.
\end{equation}
\end{lemma}

\begin{proof}
Notice that if $|f_i| \leq 1_{[N]}$ then, by Parseval and our assumption that $p \geq 2$, we have the bound
$$
\bignorm{\hat{f}_i}_p \leq \norm{f_i}_2^{2/p} \norm{f_i}_1^{(p-2)/p} \leq N^{1-\recip{p}}.
$$
Similarly, if $|f_i| \leq \nu$ then our restriction estimate assumption and the bound $\norm{\nu}_1 \ll N$ gives that $\bignorm{\hat{f}_i}_p \ll_p N^{1-\recip{p}}$.

Let $t = \floor{p} + 1$.  Using orthogonality and the trivial estimate $\bigabs{\hat{f}_i} \ll N$, we have
$$
\sum_{\vc\cdot\vx = 0} \prod_{i = 1}^s f_i(x_i) = \int_{\T} \hat{f}_{1}(c_1\alpha) \dotsm \hat{f}_s(c_s\alpha) \intd \alpha \ll N^{s-t}
\hspace{-0.2cm}
\int_{\T} \bigabs{\hat{f}_{1}(c_1\alpha) \dotsm \hat{f}_t(c_t\alpha)} \intd \alpha.
$$
By H\"older's inequality
\begin{align*}
\int_{\T} \bigabs{\hat{f}_{1}(c_1\alpha) \dotsm \hat{f}_t(c_t\alpha)} \intd \alpha & \leq \bignorm{\hat{f}_1}^{t-p}_{\infty}\bignorm{\hat{f}_1}_p^{1-(t-p)}\prod_{ 1<i \leq t} \bignorm{\hat{f}_i}_p\\
& \ll_{t, p}  \bignorm{\hat{f}_1}^{t-p}_{\infty} N^{p-1} \\
& = \brac{\bignorm{\hat{f}_1}_{\infty} N^{-1}}^{t-p} N^{t-1} .
\end{align*}
The lemma now follows with $i = 1$ on the right-hand side of \eqref{configcontroleqn}.  The general case follows on re-ordering the indices of the $f_i$.
\end{proof}

\section{The $W$-trick}\label{sec:Wtrick}

To prove Theorem \ref{thm:main} we need to construct a majorant supported on the squares which satisfies the conditions of Proposition \ref{BourbakiProp}.  We give the initial construction in this section.  The idea is that although \eqref{IntroNu} may not possess sufficient Fourier decay, by making an appropriate choice of integers $W$ and $b$, the Fourier transform of the function
$$
n \mapsto \nu(Wn + b)
$$
should exhibit the required cancellation.

Set 
\begin{equation}\label{Ww}
W := 8\prod_{2<p \leq w} p, \qquad\text{where}\qquad w := 
\sqrt{\log X}. 
\end{equation}
By the prime number theorem, there exists an absolute constant $C$ such that
\begin{equation}\label{Wbound}
W \leq \exp(C \sqrt{\log X}) \ll_\eps X^\eps,
\end{equation}
for any $\varepsilon>0$.

Let $b_1$ be a $w$-smooth number and let $b_2 \in [W]$ with $(b_2, W) = 1$.  Notice that every positive square lies in one and only one of the arithmetic progressions
$$
b_1^2( W \cdot \Z - b_2).
$$
This motivates the  definition of our $W$-tricked majorant function.
Set
$$
\sigma(b_2) : = \hash\set{ z \in [W] : z^2 + b_2\equiv 0\ \bmod W}.
$$
Then we define the function $\nu_\vb: \Z \to [0, \infty)$ by 
\begin{equation}\label{NuDef}
\nu_\vb(n) := \begin{cases} \frac{2\sqrt{Wn-b_2}}{\sigma(b_2) }  & \text{if $b_1^2(Wn - b_2) = x^2$ for some $x \in [X]$},\\
0 & \text{otherwise}.\end{cases}
\end{equation}
Note that $\sigma(b_2)\neq 0$ whenever $n\in \ZZ$ is such that 
 $b_1^2(Wn - b_2) = x^2$ for some $x \in [X]$. Indeed, in this case,  $-b_2$ must be a square modulo $W$,
which is  equivalent  to requiring  $-b_2$ to be a square modulo $8$ and a square modulo every odd prime $p\leq w$. 
There are 4 possible values of $z\bmod{8}$ such that $z^2\equiv -b_2 \bmod{8}$ and $2$ values of 
 $z\bmod{p}$ such that $z^2\equiv -b_2 \bmod{p}$ for any odd prime prime $p\leq w$.
Hence the Chinese remainder theorem 
yields
$
\sigma(b_2)=2\cdot 2^{\pi(w)}.
$
Furthermore, 
on writing
\begin{equation}\label{NDefn}
N_{\vb} := \floor{\frac{X^2}{b_1^2W}} + 1,
\end{equation}
one can check that 
$\nu_\vb$  is supported on $[N_{\vb}]$ and that
\begin{equation}\label{eqn:WhooperSwan}
\sum_n \nu_\vb(n) = N_{\vb} + O(1 + Xb_1^{-1}) = N_{\vb} + O(\sqrt{WN_{\vb}}).
\end{equation}

Given a set of integers $A$, let
$$
A_{\vb}: =\set{ n \in \Z : b_1^2(Wn -b_2) = x^2 \text{ for some } x \in A}.
$$   
We proceed by establishing the following result. 
\begin{lemma}\label{lem:MandarinDuck}
There exists an absolute constant $C$ such that the following holds for any $\delta \geq C(\log X)^{-1}$.  Let $A \subset  [X]$ with $|A| = \delta X$.  Then there exists a $w$-smooth number $b_1 \leq X^{1/2}$ and there exists $b_2 \in [W]$ with $(b_2, W) = 1$ such that
$$
\sum_n 1_{A_{\vb}}\brac{n} \nu_{\vb} (n) \gg \delta^2 N_{\vb}.
$$
\end{lemma}

Consideration of the set $A=[\delta X]$ shows that the lower bound in this estimate is optimal. 

\begin{proof}[Proof of Lemma \ref{lem:MandarinDuck}]
It follows from Rankin's trick that the 
number of integers in $[X]$ divisible by a $w$-smooth number exceeding $X^{1/2}$ is at most
\begin{align*}
\sum_{\substack{b_1>X^{1/2}\\ \text{$b_1$ $w$-smooth}}} \frac{X}{b_1} \ll
\sum_{\substack{\text{$b_1$ $w$-smooth}}} \frac{X}{b_1} \left(\frac{b_1}{X^{1/2}}\right)^{1/2} 
&\ll X^{3/4}\sum_{\substack{\text{$b_1$ $w$-smooth}}} \frac{1}{b_1^{1/2}}\\
&\ll X^{3/4}\prod_{p<w}\left(1+\frac{1}{p^{1/2}}\right)\\
&\ll X^{3/4+\varepsilon},
\end{align*}
for any $\varepsilon>0$.
It follows that
$$
\delta X = \sum_{\substack{b_1 \leq X^{1/2}\\ b_1 \text{ is $w$-smooth}}}\sum_{\substack{b_2 \in [W] \\ (b_2, W) = 1}} |A_{\vb}| + O_\varepsilon\brac{X^{3/4+\varepsilon}}.
$$
Our assumption that $\delta \geq C (\log X)^{-1}$ therefore ensures that 
$$
\delta X \ll \sum_{\substack{b_1 \leq X^{1/2}\\ b_1 \text{ is $w$-smooth}}}\sum_{\substack{b_2 \in [W] \\ (b_2, W) = 1}} |A_{\vb}| .
$$
Similarly,
$$
 \sum_{\substack{b_1 \leq X^{1/2}\\ b_1 \text{ is $w$-smooth}}}\sum_{\substack{b_2 \in [W] \\ (b_2, W) = 1}} |[X]_{\vb}| \ll X,
$$
so that
\begin{align*}
\delta \sum_{\substack{b_1 \leq X^{1/2}\\ b_1 \text{ is $w$-smooth}}}\sum_{\substack{b_2 \in [W] \\ (b_2, W) = 1}} |[X]_{\vb}| \ll \sum_{\substack{b_1 \leq X^{1/2}\\ b_1 \text{ is $w$-smooth}}}\sum_{\substack{b_2 \in [W] \\ (b_2, W) = 1}} |A_{\vb}|.
\end{align*}
Applying the pigeon-hole principle, we conclude that there exists a $w$-smooth number $b_1 \leq X^{1/2}$ and $b_2 \in [W]$ with $(b_2, W) =1$ such that  
$$
\delta |[X]_{\vb}| \ll |A_{\vb}|.
$$
For each $z \in [W]$ with $z^2+b_2 \equiv 0 \bmod W$, write $A_{\vb, z}$ for the set of positive integers $x \equiv z \bmod W$ such that $x^2 = Wn - b_2$ for some $n \in A_{\vb}$.  Together these sets satisfy
$$
\sum_{z^2 +b_2 \equiv 0 \bmod W} |A_{\vb, z}| = |A_{\vb}| \gg \delta |[X]_{\vb}|.
$$

For any $T>0$ 
we claim that there are at most $T+1$ values of $x \in A_{\vb, z}$  satisfying $x \leq WT$.  
This follows on noting that elements 
of $A_{\vb, z}$ take the shape $z+Wy$ for suitable $y$. But there are clearly at most $T-z/W+1$ choices of $y$ satisfying $y\leq T-z/W$, from which the claim follows.
Applying the claim with $T=\trecip{2}|A_{\vb,z}|$, we deduce that for at least $\trecip{2} |A_{\vb, z}| -1$ values of $x \in A_{\vb, z}$  we have $x > \trecip{2}W|A_{\vb,z}|$. It follows from \eqref{NuDef}
that
\begin{align*}
\sum_n 1_{A_{\vb}}(n)\nu_{\vb}(n)&  = \frac{1}{\sigma(b_2)}\sum_{z^2 +b_2 \equiv 0 \bmod W} \sum_{{\ } x \in  A_{\vb, z}} 2x\\
&  \geq  \frac{2}{\sigma(b_2)}\sum_{z^2 +b_2 \equiv 0 \bmod W} 
(\trecip{2} |A_{\vb, z}|-1)\trecip{2} W |A_{\vb, z}|.
\end{align*}
An application of the  Cauchy--Schwarz inequality now shows that the right hand side is
\begin{align*}
&>\frac{W}{2}\brac{ \brac{ \frac{1}{\sigma(b_2)}\sum_{z^2 +b_2 \equiv 0 \bmod W} |A_{\vb, z}|}^2 - 2\frac{1}{\sigma(b_2)}\sum_{z^2 +b_2 \equiv 0 \bmod W} |A_{\vb, z}|} \\
& = \frac{W}{2} \brac{ \brac{\frac{|A_{\vb}|}{\sigma(b_2)}}^2 - 2\frac{|A_{\vb}|}{\sigma(b_2)}} .
\end{align*}
Recalling that  $b_1 \leq X^{1/2}$ and  $W \ll_\eps X^\eps$, 
provided we take $C$ sufficiently large in our lower bound $\delta \geq C(\log X)^{-1}$, 
we may proceed under the assumption that 
 $|A_{\vb}| \geq 4\sigma(b_2)$. But then it follows that 
$$
\brac{ \brac{\frac{|A_{\vb}|}{\sigma(b_2)}}^2 - 2\frac{|A_{\vb}|}{\sigma(b_2)}} \geq \recip{2} \brac{\frac{|A_{\vb}|}{\sigma(b_2)}}^2.
$$
Since we have 
$$
|[X]_{\vb}| = \sigma(b_2)\brac{ \frac{ X}{b_1W} + O(1)},
$$
it finally follows that 
$$
\sum_n 1_{A_{\vb}}(n)\nu_{\vb}(n) \gg \delta^2 \brac{\frac{X^2}{b_1^2 W}} \gg \delta^2 N_{\vb},
$$
as required to complete the proof of the lemma.
\end{proof}

\section{Exponential sum estimates}\label{sec:ExpSums} 

Henceforth we fix a choice of $\vb$ afforded by Lemma \ref{lem:MandarinDuck} and denote both $\nu_{\vb}$ and $N_{\vb}$ by $\nu$ and $N$, respectively.  
For $z \in [W]$ such that $z^2 +b_2 \equiv 0\bmod{W}$, write
$$
S_q(a,z) := \sum_{r=1}^q e\brac{\frac{a(Wr^2 + 2zr+\frac{z^2+b_2}{W})}{q}} \quad \text{and}\quad
I(\beta) := \int_0^{N} e\bigbrac{\beta t }\intd t.
$$
Bearing these quantities in mind, we may now establish the following result.

\begin{lemma}[Major arc asymptotic]\label{lem:MajorArc}
Suppose that $\norm{q\alpha} = |q\alpha - a|$ for 
 $q,a\in \ZZ$ with $q>0$.  Then
\begin{equation*}
\begin{split}
 \hat{\nu}(\alpha) =~& \frac{1}{\sigma(b_2)} 
  \hspace{-0.2cm}
 \sum_{\substack{z \in [W]\\ z^2 +b_2 \equiv 0 \bmod W}} 
 \hspace{-0.3cm}
 q^{-1} S_q(a, z)I(\alpha - \tfrac{a}{q})  
 + O\brac{\sqrt{NW}(q + N \norm{q\alpha})}.
\end{split}
\end{equation*}
\end{lemma}

\begin{proof}  Recalling \eqref{FourierDef} and \eqref{NuDef}, we break the Fourier transform into congruence classes modulo $W$ to give
\begin{align*}
\hat{\nu}(\alpha) &= \frac{1}{\sigma(b_2)}\sum_{\substack{ x \leq Xb_1^{-1}\\ x^2 +b_2 \equiv 0 \bmod W}}  2x\ e\bigbrac{\alpha (x^2+ b_2)/W}\\
& = \frac{e(\alpha b_2/W)}{\sigma(b_2)}\sum_{\substack{z \in [W]\\ z^2 +b_2 \equiv 0 \bmod W}}\sum_{z + Wy \leq Xb_1^{-1}}  2(z+Wy)\ e\bigbrac{\alpha (z+Wy)^2/W}.
\end{align*}
We now write  $\alpha =  \frac{a}{q}+ \beta$ and break the sum over $y$ into residue classes modulo $q$. This gives
\begin{align*}
\sum_{z + Wy \leq Xb_1^{-1}}  &2(z+Wy)\ e\bigbrac{\alpha (z+Wy)^2/W}\\
  &= \sum_{r=1}^q  \sum_{\substack{z + Wy \leq Xb_1^{-1}\\ y \equiv r \bmod q}} 2(z+Wy)\ e\bigbrac{(\tfrac{a}{q} + \beta)
  (z+Wy)^2/W} \\
&=\sum_{r=1}^q e\bigbrac{a(z+Wr)^2/(qW)} \sum_{z + Wr + Wq x \leq Xb_1^{-1}} 
\phi(x),
\end{align*}
where
$
\phi(x) := 2(z + Wr + Wq x)\ e\bigbrac{\beta(z + Wr + Wq x)^2/W}.
$

We shall use Euler--Maclaurin to estimate the inner sum over $x$ (in the form \cite[Eq.~(4.8)]{vaughan97}, for example).
Taking $Y_0=(Xb_1^{-1}-(z+Wr))/(Wq)$ and $X_0=-(z+Wr)/(Wq)$, this yields
$$
\sum_{X_0< x\leq Y_0} \phi(x)=\int_{X_0}^{Y_0} \phi(t)\intd t -\left[\phi(t) \psi(t)\right]^{Y_0}_{X_0} +\int_{X_0}^{Y_0} \phi'(t) \psi(t)\intd t,
$$ 
where 
 $\psi(t)=\{t\}-\tfrac{1}{2}=O(1)$.  Next we make the change of variables 
$u=(z+Wr+Wqt)^2/W$, under which the first term becomes
$$
\int_{X_0}^{Y_0} \phi(t)\intd t=
q^{-1}\int_0^{X^2/(b_1^2W)} e(\beta t)\intd t=
q^{-1}(I(\beta)+O(1)).
$$
Next, on noting that $|\phi(t)|\leq 2Xb_1^{-1}$ on $(X_0,Y_0]$ and 
$$
\sup_{t\in (X_0,Y_0]} |\phi'(t)| \ll W(q+N\|q\alpha\|),
$$
we find that the remaining two  terms are
\begin{align*}
&\ll \frac{X}{b_1} +\frac{X}{b_1Wq} W(q+N\|q\alpha\|)\ll \frac{\sqrt{NW}}{q}(q+N\|q\alpha\|).
\end{align*}
Inserting this into our previous estimate, we  conclude that
\begin{equation*}
\begin{split}
 \hat{\nu}(\alpha) =~& \frac{e(\alpha b_2/W)}{\sigma(b_2)} \sum_{\substack{z \in [W]\\ z^2 +b_2 \equiv 0 \bmod W}} q^{-1} I(\alpha-\tfrac{a}{q}) 
 \sum_{r=1}^q e\bigbrac{a(z+Wr)^2/(qW)}
  \\
 &\quad + O\brac{\sqrt{NW}(q + N \norm{q\alpha})}.
\end{split}
\end{equation*}
Now the inner exponential sum over $r$ is equal to $S_q(a,z)e(-ab_2/(qW))$. The lemma therefore follows on noting that 
\begin{align*}
e(\alpha b_2/W)e(-ab_2/(qW))I(\alpha-\tfrac{a}{q})
&=\int_0^N e\left((\alpha-\tfrac{a}{q}) (t+\tfrac{b_2}{W})\right)\intd t\\ &= I(\alpha-\tfrac{a}{q})+O(1),
\end{align*}
since $b_2/W\in [0,1]$.
\end{proof}

We need the following basic estimates for the quantities featuring in Lemma \ref{lem:MajorArc}.

\begin{lemma}\label{lem:estimate-SI}
For $(a,q) = 1$ we have 
$$
|S_q(a,z)| \leq 2 \sqrt{q} \quad \text{ and } \quad 
I(\beta)\leq \min\left\{ N, \frac{1}{ \|\beta\|} \right\}.  
$$
\end{lemma}

\begin{proof}
The estimate for $I(\beta)$ is standard and so we focus on
$S_q(a,z)$.
Let $h=(q,W)$ and put $q=hq'$ and $W=hW'$, with $(q',W')=1$.
We then write $r=r_1+q'r_2$ for $r_1$ running modulo $q'$ and $r_2$ running modulo $h$, finding that
\begin{align*}
S_q(a,z)
&=
e\brac{\frac{a(z^2+b_2)}{qW}} 
\sum_{r=1}^q e\brac{\frac{a(Wr^2 + 2zr)}{q}} \\
&=
e\brac{\frac{a(z^2+b_2)}{qW}} 
\sum_{r_1=1}^{q'} e\brac{\frac{a(W'r_1^2+2zr_1/h)}{q'}} 
\sum_{r_2=1}^{h}
e\brac{\frac{2azr_2}{h}}.
\end{align*}
The inner sum vanishes unless $h\mid 2az$. But this happens if and only if $h\mid 2$, since 
$z^2+b_2\equiv 0\bmod{W}$ and 
$(a,q)=(b_2,W)=1$.   Hence  $S_q(a,z)=0$ if $h\nmid 2$ and 
\begin{equation}\label{eq:flamingo}
S_q(a,z)
=
h\, e\brac{\frac{a(z^2+b_2)}{qW}} 
\sum_{r_1=1}^{q'} e\brac{\frac{a(W'r_1^2+2zr_1/h)}{q'}}.
\end{equation}
if $h\mid 2$.
Assuming that $h\mid 2$, 
we see that the desired bound for 
$S_q(a,z)$  is  an immediate consequence of the basic estimate
$$
\abs{\sum_{r=1}^{s} e\brac{\frac{br^2+cr}{s}}}\leq \sqrt{2s},
$$
which is valid for any positive integer $s$ and any integers $b,c$ with $b$ coprime to $s$.
The latter estimate follows on invoking the squaring and differencing approach found in the standard  analysis of Weyl sums. 
\end{proof}

Building on the proof of this result  we may establish the following. 

\begin{lemma}\label{lem:Sq}
If $q>1$ and  $q$ is $w$-smooth then
$$
\sum_{\substack{z \in [W]\\ z^2 +b_2 \equiv 0 \bmod W}}q^{-1}S_q(a,z)=0.
$$ 
\end{lemma}

\begin{proof}
Suppose that $q$ is  $w$-smooth and bigger than $1$.
We saw in the proof of the previous result that $S_q(a,z)=0$ unless $(q,W)\mid 2$. 
This implies that $q=2^h$ for some $h\geq 1$, since $q$ is $w$-smooth.
Since $8\mid W$ we must have $\min\{h,3\}=2$, whence 
 $q=2$ and  $a=1$. Substituting these values into \eqref{eq:flamingo}, we deduce that 
$$
S_2(1,z)=2e\left(\frac{z^2+b_2}{2W}\right).
$$
Let $\tilde W$ be the odd integer such that $W=8\tilde W$.
We now introduce the sum over $z$, writing $z=u+4\tilde W v$ for $u$ running modulo $4\tilde W$ and $v$ running modulo $2$. In particular, we have 
$$
e\left(\frac{z^2+b_2}{2W}\right)=e\left(\frac{u^2+b_2}{2W}\right)e\left(\frac{uv}{2}\right)
$$
and 
$z^2 +b_2 \equiv 0\bmod{W}$ if and only if $u^2 +b_2 \equiv 0 \bmod{W}$. 
We must have $2\nmid u$ since $z^2+b_2\equiv 0\bmod{W}$ and $(b_2,W)=1$. But then 
$$
\sum_{\substack{z \in [W]\\ z^2 +b_2 \equiv 0 \bmod W}}q^{-1}S_q(a,z)=
\sum_{\substack{u \bmod{4\tilde W}\\ u^2 +b_2 \equiv 0 \bmod W}}
e\left(\frac{u^2+b_2}{2W}\right)
\sum_{v\bmod{2}}
e\left(\frac{uv}{2}\right)=0,
$$ 
as required.
\end{proof}

Recalling  \eqref{eqn:WhooperSwan}, we have  the trivial estimate $\hat{\nu}(\alpha)\ll N$.  Our next result 
 is a standard Weyl bound, showing how one can beat the trivial estimate when $\alpha$ is not well approximated by a rational with small denominator.

\begin{lemma}[Weyl bound] 
\label{lem:Weyl}
 Suppose that $b_1 W \leq X$ and $\norm{q\alpha} = |q\alpha - a|$ for  coprime $a, q \in \Z$ with $q>0$.  Then
\begin{align*}
\abs{\hat{\nu}(\alpha)} \ll~& N \sqrt{W\log N}\\
&\times 
\brac{\brac{WN}^{-1/2} +  \norm{q\alpha} + qN^{-1} + \min\set{q^{-1},(\norm{q\alpha}N )^{-1}}}^{1/2}.
\end{align*}
\end{lemma}

\begin{proof}
As in the proof of Lemma \ref{lem:MajorArc}, we have
$$
|\hat{\nu}(\alpha)| \leq 
\frac{1}{\sigma(b_2)}
 \sum_{\substack{z \in [W]\\ z^2 +b_2 \equiv 0 \bmod W}}\abs{\sum_{z + Wy \leq Xb_1^{-1}}  2(z+Wy)\ e\bigbrac{\alpha (Wy^2 + 2zy)}}.
$$
By partial summation and the assumption that $b_1 W \leq X$, one sees that
\begin{multline}\label{eqn:PartialSum}
\abs{\sum_{z + Wy \leq Xb_1^{-1}}  2(z+Wy)\ e\bigbrac{\alpha (Wy^2 + 2zy)}}\\ \ll Xb_1^{-1} \sup_{ Y \leq X/(b_1W)} \abs{\sum_{0 \leq y \leq Y}   e\bigbrac{\alpha (Wy^2 + 2zy)}}.
\end{multline}

Let us denote temporarily by $S(Y)$ the inner sum in \eqref{eqn:PartialSum}.  
This is an exponential sum over a polynomial with unbounded height and we  require a version of Weyl's inequality applicable to this situation. 
A standard Weyl differencing argument gives that $ S(Y)$ is
\begin{align*} 
 \ll Y(\log Y)^{1/2}
\brac{Y^{-1} + \norm{q\alpha} W + qY^{-2} + \min\set{Wq^{-1}, (\norm{q\alpha}Y^2)^{-1}}}^{1/2}.
\end{align*}
(cf.  \cite[Lemma 3.3]{browningprendiville14} with  $(d,n,\sigma)=(2,1,0)$ and $(P,H)=(Y,W)$ in the notation therein.) 
One can check that this bound for $S(Y)$ is increasing in $Y$, and so it is maximised when $Y = X/(b_1W)$.  Since $\sqrt{N/W} \ll X/(b_1W) \ll \sqrt{N/W}$, we conclude that
the right hand side of 
\eqref{eqn:PartialSum} is 
$$
\ll Xb_1^{-1}\cdot \sqrt{N\log N \brac{(NW)^{-1/2} + \norm{q\alpha}  + qN^{-1} + \min\set{q^{-1}, (\norm{q\alpha}N)^{-1}}}}.
$$
The statement of the lemma is now obvious.
\end{proof}

Note that $b_1W\leq X$ since $b_1\leq X^{1/2}$ and $W\ll_\varepsilon X^{\varepsilon}$ by \eqref{Wbound}, and so the hypothesis of the preceding result is met. 
We end this section by showing that our majorant function has suitable Fourier decay.

\begin{lemma}\label{eq:bounded}
We have the estimate
$$
\|\hat \nu - \hat 1_{[N]}\|_\infty \ll  \frac{N}{\sqrt{w}}
$$
\end{lemma}

\begin{proof}
Let $\tau=\frac{1}{100}.$  We define a set of major arcs $\mathfrak{M}$ to be the union of 
all intervals
$$
\mathfrak{M}(r,b)=\{\alpha\in \T: |\alpha-b/r|\leq N^{-1+\tau}\}
$$
for coprime integers $b,r$ such that $0\leq b< r\leq N^\tau$.  Now let $\alpha\not\in \mathfrak{M}.$ By Dirichlet's approximation theorem we can find coprime integers $a,q$ such that 
$q\leq N^{\tau}$ and 
$|\alpha-\frac{a}{q}|\leq q^{-1}N^{-\tau}$.
Since $\alpha$ is  not on the major arcs we must have 
 $|q\alpha-a|>qN^{-1+\tau}$. It therefore  follows from Lemma \ref{lem:Weyl}
that
\begin{equation}\label{minorarcbound}
\hat \nu(\alpha)\ll N^{1-\tau/2}\sqrt{W\log N}.
\end{equation}
Since 
$$
\hat 1_{[N]}(\alpha)=\sum_{1\leq n \leq N} e(\alpha n)\ll \frac{q}{\|q\alpha\|}\ll N^{1-\tau}, 
$$
and $\sqrt{W\log N}\ll N^{\tau/4}$ by \eqref{Wbound}, 
we conclude that 
$$
\abs{\hat \nu(\alpha) - \hat 1_{[N]}(\alpha)}\ll N^{1-\tau/4}
$$
for any $\alpha\not\in \mathfrak{M}.$ 

Next let $\alpha\in \mathfrak{M}(q,a)$ for coprime $a,q\in \ZZ$ with  $0\leq a< q\leq N^\tau$.
According to Lemma \ref{lem:MajorArc} we have 
\begin{equation}\label{eq:parrot}
 \hat{\nu}(\alpha) = \frac{1}{\sigma(b_2)} \sum_{\substack{z \in [W]\\ z^2 +b_2 \equiv 0 \bmod W}} q^{-1} S_q(a, z)I(\alpha - \tfrac{a}{q})  
 + O\brac{N^{1/2+3\tau}},
\end{equation}
since
$
q + N \norm{q\alpha}\ll q+N\cdot qN^{-1+\tau}\ll qN^{\tau}\ll N^{2\tau}
$
and $W\leq N^\tau$. 
In a similar fashion (see  \cite[Lemma 2.7]{vaughan97}, for example), one finds that  
$$
\hat 1_{[N]}(\alpha)=
\begin{cases}
O(N^{2\tau}) &\text{if  $q>1$,}\\
I(\alpha)+O(N^{2\tau}) & \text{if $q=1$.}
\end{cases}
$$
The 
 main term in \eqref{eq:parrot} has absolute value 
 $$
 \leq \frac{N}{\sigma(b_2)} \left|\sum_{\substack{z \in [W]\\ z^2 +b_2 \equiv 0 \bmod W}} q^{-1} S_q(a, z)\right|,
  $$
since $|I(\beta)|\leq N$.
Lemma \ref{lem:Sq} implies that this vanishes unless either $q=1$ 
 or $q$ is not $w$-smooth. In the latter case we may conclude that $q>w$, whence an application of 
Lemma \ref{lem:estimate-SI} shows that the 
 main term has absolute value at most
$
2N/{\sqrt{q}} \leq 2N/\sqrt{w}.
$
It follows that 
$$
\abs{\hat \nu(\alpha) - \hat 1_{[N]}(\alpha)}\ll 
\frac{N}{\sqrt{w}}+  N^{1/2+3\tau}+ N^{2\tau}\ll \frac{N}{\sqrt{w}} 
$$
for any $\alpha\in \bigcup_{q\neq 1}\bigcup_{a\bmod{q}}\mathfrak{M}(q,a)$.

Finally, when $q=1$ and $\alpha\in \mathfrak{M}(1,0)$, we deduce from \eqref{eq:parrot} that 
$$
 \hat{\nu}(\alpha) =  I(\alpha )  
 + O\brac{N^{1/2+3\tau}}, 
$$
since $S_1(0,z)=1$.
Hence 
$$
\abs{\hat \nu(\alpha) - \hat 1_{[N]}(\alpha)}\ll 
 N^{1/2+3\tau}
$$
for any  $\alpha\in \mathfrak{M}(1,0)$.
This completes the proof of the lemma.
\end{proof}

\section{The restriction estimate}\label{sec:restriction}

We require an analogue of Bourgain's restriction estimate \cite{bourgain89} for the  $W$-tricked quadratic exponential sum. The proof of the following result closely follows the  argument in \cite[\S 4]{bourgain89}. 

\begin{proposition}[Restriction bound]\label{thm:RestrictionBound}  For any $p > 4$ there exists an absolute constant $C_p$ such that for any function $\phi : \ZZ \to \CC$ satisfying
$|\phi| \leq \nu$ we have the bound
$$
\int_\T |\hat{\phi}(\alpha)|^p \intd \alpha \leq C_p N^{p-1}.
$$
\end{proposition}

The  
following result is concerned with calculating the 
fourth moment of $\hat \phi$, since it plays an important role in the proof of Proposition~\ref{thm:RestrictionBound}. 

\begin{lemma}\label{lem:4th-moment}
There exists an absolute constant $C$ such that for any function $\phi : \ZZ \to \CC$ satisfying
$|\phi| \leq \nu$ we have the bound
$$
\int_\T |\hat{\phi}(\alpha)|^4 \intd \alpha \leq N^{3 + C/\log\log N}
$$
\end{lemma}
\begin{proof}
The left hand side is equal to 
\begin{align*}
\|\hat \phi\|_4^4 
&=\left\| \sum_{m,n} \phi(m)\bar{\phi}(n) e\left(\alpha(m-n)\right) \right\|_2^2
\leq \sum_{|k|\leq N} \left| \sum_{\substack{m,n\\ m-n=k}}\phi(m)\bar{\phi}(n)\right|^2.
\end{align*}
The contribution to the sum over $k$ from $k=0$ is 
\begin{align*}
\leq \left(\sum_n |\phi(n)|^2\right)^2
\leq \left(\sum_n \nu(n)^2\right)^2
\ll \left(Xb_1^{-1}N\right)^2\ll WN^3,
\end{align*}
by  \eqref{NuDef}, \eqref{NDefn} and \eqref{eqn:WhooperSwan}.
To handle the contribution from $k\neq 0$, 
the definition of $\nu$ makes it clear that 
$$
\sum_{\substack{m,n\\ m-n=k}}\phi(m)\bar{\phi}(n)=
\sum_{\substack{x,y\leq X/b_1\\ x^2\equiv y^2\equiv -b_2\bmod{W}\\ x^2-y^2=Wk}}\phi\left(\frac{x^2+b_2}{W}\right)\bar{\phi}\left(\frac{y^2+b_2}{W}\right)
$$
For fixed non-zero $k$ there are clearly at most $d(k)$ choices for $x,y$. Hence Cauchy--Schwarz gives
\begin{align*}
\|\hat \phi\|_4^4 
\leq  \sum_{0<|k|\leq N} d(k) 
\hspace{-0.2cm}
\sum_{\substack{x,y\leq X/b_1\\ x^2\equiv y^2\equiv -b_2\bmod{W}\\ x^2-y^2=Wk}} 
\hspace{-0.2cm}
\nu\left(\frac{x^2+b_2}{W}\right)^2\nu\left(\frac{y^2+b_2}{W}\right)^2 +O(WN^3)
\end{align*}
The standard estimate for the divisor function yields  $
d(k) \leq N^{O(1/\log\log N)},
$ for $0<|k|\leq N$. 
 Hence 
\begin{align*}
\|\hat \phi\|_4^4 
&\leq N^{O(1/\log\log N)}\left(
\sum_{\substack{x\leq X/b_1\\ x^2+b_2 \equiv 0\bmod{W}}} \nu\left(\frac{x^2-b_2}{W}\right)^2 \right)^2 +O(WN^3)\\
&\leq N^{O(1/\log\log N)}\left(Xb_1^{-1}N \right)^2 +O(WN^3)\\ 
&\leq WN^{3+ O(1/\log\log N)},
\end{align*}
as before.
Finally, since  $X/b_1 \geq X^{1/2}$, it follows from 
 \eqref{Wbound} that 
$W$ is at most $N^{O(1/\log\log N)}$.
This concludes the proof of the lemma. 
\end{proof}

Define the region
$$
\cR_\delta=\left\{\alpha\in \TT : |\hat \phi(\alpha)|>\delta N \right\},
$$
for any  $\delta\in (0,1)$.
Proposition \ref{thm:RestrictionBound} is a consequence of the following result.

\begin{lemma}\label{lem:measure}
For any $\delta\in (0,1)$ and any $\eps>0$ there exists a constant $C_\eps$ depending only $\eps$ such that 
$$
\meas( \cR_\delta) \leq \frac{C_\eps}{\delta^{4+\eps} N}.
$$
\end{lemma}

Although standard, let us spell out how Lemma \ref{lem:measure} implies Proposition~\ref{thm:RestrictionBound}. 
Breaking into dyadic intervals it follows from Lemma \ref{lem:measure}  that 
\begin{align*}
\|\hat \phi\|_p^p 
&\leq 
\sum_{j\geq 0} \left(\frac{N}{2^{j-1}}\right)^p
\meas\left\{\alpha\in \TT: 
\frac{N}{2^j}<| \hat \phi(\alpha)|\leq \frac{N}{2^{j-1}}\right\}\\
&\leq C_\eps 2^p N^{p-1}\sum_{j\geq 0} (2^{4+\eps-p})^j.
\end{align*}
Taking $\eps = \eps(p) >0$ sufficiently small we can ensure that this geometric series converges, which thereby gives Proposition \ref{thm:RestrictionBound}. 

We begin the proof of Lemma \ref{lem:measure} with a cruder upper bound based on 
our fourth moment estimate. Thus Lemma \ref{lem:4th-moment} implies that there exists an absolute constant $C > 0$ such that
$$
(\delta N)^4 \meas (\cR_\delta) \leq \|\hat \phi\|_4^4 \leq  N^{3 + C/\log\log N}, 
$$
whence 
$$
\meas (\cR_\delta)  \leq \frac{N^{C/\log\log N}}{\delta^4 N}.
$$
In particular,  Lemma \ref{lem:measure} follows from this bound unless 
\begin{equation}\label{eq:delta-lower}
\delta>N^{-C(\eps\log \log N)^{-1}},
\end{equation}
as we henceforth  assume.

As in the proof of Lemma \ref{eq:bounded}, let $\tau=\tfrac{1}{100}$.  To each fraction $a/q$, with  $0\leq a<q\leq N^\tau$ and  $(a,q)=1$, we associate the major arc
$$
\mathfrak{M}(q,a)=\{\alpha\in \TT: |\alpha-a/q|\leq N^{-1+\tau}\}.
$$
We let $\mathfrak{M}$ denote the union of all major arcs. 
As in \eqref{minorarcbound} we have 
\begin{equation}\label{eq:nu-minor}
\hat \nu(\alpha) \ll N^{1-\tau/4}   \quad \text{if $\alpha\not\in \mathfrak{M}$.}
\end{equation}
On the major arcs we can conclude from \eqref{eq:parrot} and Lemma \ref{lem:estimate-SI} that 
\begin{equation}\begin{split}
\label{eq:nu-major}
\abs{\hat{\nu}(\alpha)} 
&\leq  \frac{1}{\sigma(b_2)} \sum_{\substack{z \in [W]\\ z^2 +b_2 \equiv 0 \bmod W}} q^{-1} |S_q(a, z)I(\alpha - \tfrac{a}{q}) | + O(N^{1/2+2\tau})\\
&\leq \frac{2}{\sqrt{q}}\min\left\{N,\|\alpha-\tfrac{a}{q}\|^{-1}\right\}  +O(N^{1/2+2\tau})   
\end{split}
\end{equation}
if $\alpha\in \mathfrak{M}(q,a)$ for some $a,q$ satisfying 
$0\leq a< q\leq N^\tau$ and  $(a,q)=1$.

Returning to our analysis of  $\hat \phi$, 
let $\theta_1,\dots,\theta_R$ be $1/N$-spaced points in $\TT$ such that 
$
|\hat \phi(\theta_r) |>\delta N,$ for $1\leq r\leq R$.
In order to prove Lemma \ref{lem:measure} it will suffice to show that
\begin{equation}\label{eq:flight}
R\leq \frac{C_\eps}{\delta^{4+\eps}}.
\end{equation}
To see this,  
let $R\geq 1$ be the largest integer such that there exists a sequence of $1/N$-spaced points 
$\theta_1,\dots,\theta_R$ in $\mathcal{R}_\delta$. Then every $\theta\in \mathcal{R}_\delta$ necessarily satisfies 
$\|\theta-\theta_r\|<1/N$ for some $r$. This implies that the measure of $\mathcal{R}_\delta$ does not exceed $2R/N$, as required.

To prove \eqref{eq:flight}, we 
 let $a_n\in \CC$ be such that $|a_n|\leq 1$ and 
$\phi(n)=a_n\nu(n)$ for integers  $1\leq n\leq N$. Furthermore,  
 let $c_1,\dots,c_R\in \CC$ be such that $|c_r|=1$ and 
$$
c_r\hat \phi (\theta_r) =|\hat \phi(\theta_r)| 
$$
for $1\leq r\leq R$.  Then it follows from the Cauchy--Schwarz inequality that
\begin{align*}
\delta^2 N^2 R^2
&\leq \left(\sum_{1\leq r\leq R} |\hat \phi(\theta_r)|\right)^2\\
&= \left(\sum_{1\leq r\leq R} c_r\sum_n a_n \nu(n) e(n\theta_r)\right)^2\\
&\ll N \sum_{n} \nu(n)   \left|\sum_{1\leq r\leq R} c_re(n\theta_r)\right|^2, 
\end{align*}
so that 
\begin{align*}
\delta^2 N R^2
&\ll \sum_{1\leq r,r'\leq R} \left|\hat \nu(\theta_r-\theta_{r'})\right|.
\end{align*}

Now let $\gamma>2$ be fixed. It follows from H\"older's inequality that 
$$
\delta^{2\gamma} N^{\gamma} R^2\ll 
\sum_{1\leq r,r'\leq R} \left|\hat \nu(\theta_r-\theta_{r'})\right|^\gamma.
$$
Let us put $\theta=\theta_r-\theta_{r'}$ for given $r\neq r'$.
It follows from
 \eqref{eq:delta-lower} and $\eqref{eq:nu-minor}$ that 
$|\hat \nu(\theta)|$ is negligible when $\theta \not\in \mathfrak{M}$. When 
$\theta \in \mathfrak{M}$ we use the major arc approximation recorded in \eqref{eq:nu-major}
to conclude that 
$$
\hat \nu(\theta) \ll \frac{N}{{\sqrt{q}}}(1+N \|\theta-\tfrac{a}{q}\|)^{-1}  +O(N^{1/2+2\tau}), 
$$
for some $a,q$ satisfying 
$0\leq a< q\leq N^\tau$ and  $(a,q)=1$. 
Now let $Q=\delta^{-5}$. If $q\geq Q$ then we simply take the upper bound
$$
\frac{N}{{\sqrt{q}}}(1+N \|\theta-\tfrac{a}{q}\|)^{-1}
\ll\delta^{5/2} N.
$$
Recalling that $\delta<1$ satisfies the lower  bound in \eqref{eq:delta-lower}, 
this leads to the conclusion that 
$$
\delta^{2\gamma} R^2\ll  
 \sum_{q\leq Q}\sum_{\substack{a\bmod{q}\\ (a,q)=1}}\
\sum_{\substack{1\leq r,r'\leq R}}
\frac{q^{-\gamma/2}}{(1+N\|\theta_r-\theta_{r'}-\tfrac{a}{q}\|)^\gamma}.
$$
Hence
$$
\delta^{2\gamma} R^2\ll 
\sum_{\substack{1\leq r,r'\leq R}}
G(\theta_r-\theta_{r'}),
$$
where 
$$
G(\alpha):= \sum_{q\leq Q}\sum_{\substack{a\bmod{q}\\ (a,q)=1}}\
q^{-\gamma/2}(1+N|\sin (\alpha-\tfrac{a}{q})|)^{-\gamma}
$$
and
we recall that $\theta_1,\dots,\theta_R$ is a sequence of $1/N$-spaced points in $\TT$.

We have finally arrived at  exactly the same expression as that found in \cite[Eq.~(4.16)]{bourgain89}, but with $N^2$ replaced by $N$.  Bourgain's argument therefore  directly 
applies
 and leads us to the desired bound \eqref{eq:flight} for $R$.

\section{The $K$-trivial count}\label{sec:KTriv}

The purpose of this section is to show that $\nu = \nu_{\vb}$ saves something over the trivial estimate when counting solutions to \eqref{MainEqn} contained in a proper subspace.

\begin{lemma}\label{PilaBound}
Let $\varepsilon>0$ and let $K \subset \Q^s$ denote a finite union of at most $k$ proper subspaces of the hyperplane 
$$
c_1y_1 + \dots + c_s y_s = 0.
$$
Then the number of $\vx \in [X]^s$ with $(x_1^2, \dots, x_s^2) \in K$ is  $O_{k, s,\varepsilon}(X^{s-\frac{5}{2}+\varepsilon}).$
\end{lemma}

\begin{proof}  
The set $K$ is a union of at most $k$ subspaces of the form 
$$
\set{\vy \in \Q^s : \vc \cdot \vy = 0 \text{ and } \vd \cdot \vy = 0},
$$
where $\vd \in \Q^s$ is a  fixed $s$-tuple which is not proportional to $\vc$.  
Our task is therefore to count solutions $\vx \in [X]^s$ to the simultaneous quadrics 
$$
c_1 x_1^2 + \dots + c_s x_s^2 = d_1 x_1^2 + \dots + d_s x_s^2 = 0.
$$
This pair of equations defines a projective subvariety $V\subset \mathbb{P}_\Q^{s-1}$ of codimension $2$.  Let $\Lambda$ be any  $k$-plane contained in $V$. Then $\Lambda$ is contained in the 
 non-singular quadric hypersurface
 $
 c_1x_1^2+\dots+c_s x_s^2=0,
 $ 
 in $\mathbb{P}_\Q^{s-1}$. It follows that $k\leq (s-2)/2$ (see Harris \cite[Chap.~22]{harris2013algebraic}, for example). Since $s\geq 5$ we conclude that $V$ does not contain any linear component of maximal dimension. 
Thus we may apply work of Pila \cite{pila1995density} to each of the irreducible components of $V$, 
concluding that 
the number of  solutions $\vx \in [X]^s$ to the equations defining $V$ is 
$O_{s,\varepsilon}(X^{s-3+\frac{1}{2}+\varepsilon})$, for any $\varepsilon>0$.
This completes the proof of the lemma.
\end{proof}

\begin{corollary}\label{cor:8.2}
Let $K \subset \Q^s$ denote a finite union of at most $k$ proper subspaces of the hyperplane \eqref{MainEqn}.  Then, in the terminology of Definition~\ref{saves eta}, the majorant $\nu = \nu_{\vb}$ saves $1/5$ on the $K$-trivial solutions.
\end{corollary}

\begin{proof}
We wish to estimate
$
\sum_{\vn \in K} \prod_{i=1}^s \nu(n_i).
$
If $\vn$ is counted by this sum, then on setting $x_i = \sqrt{Wn_i - b_2}$, we obtain integers $x \in [X/b_1]$ such that $(x_1^2, \dots, x_s^2) \in K$.  Combining  the crude estimate $\nu(n_i) \ll X/b_1$ with  Lemma \ref{PilaBound}, we deduce that
\begin{align*}
\sum_{\vx \in K} \prod_{i=1}^s \nu(x_i) & \ll_\eps (Xb_1^{-1})^{2s - \frac{5}{2} + \eps}
\ll_\eps (NW)^{s-\frac{5}{4}+\frac{\eps}{2}}\ll_\eps N^{s-\frac{5}{4}+\eps},
\end{align*}
since $W^{s-\frac{5}{4}+\frac{\eps}{2}} \ll_\eps N^{\frac{\eps}{2}}$ by \eqref{Wbound}.
This completes the proof.
\end{proof}

\section{Proof of Theorem \ref{thm:main}}\label{sec:MainThmProof}

Let $A \subset [X]$ with $|A| = \delta X$ and such that $A$ contains only $K$-trivial solutions to \eqref{eqn:heron}.  We may assume that $\delta \geq C(\log X)^{-1}$ where $C$ is the absolute constant appearing in Lemma \ref{lem:MandarinDuck}, otherwise the proof of Theorem \ref{thm:main} is complete.  Under this assumption, Lemma \ref{lem:MandarinDuck} yields the existence of a $w$-smooth number $b_1 \leq X^{1/2}$ and $b_2 \in [W]$ with $(b_2, W) = 1$ such that
$$
\sum_n 1_{A_{\vb}}\brac{n} \nu_{\vb} (n) \gg \delta^2 N_{\vb},
$$
where
$
A_{\vb} =\set{ n \in \Z : b_1^2(Wn -b_2) = x^2 \text{ for some } x \in A}
$   
and $N_\vb$ is given by \eqref{NDefn}. 
Given such a choice of $\vb$, we have established that
\begin{itemize}
\item $\nu_{\vb}$ satisfies a restriction estimate at any exponent $p > 4$ (Prop.~\ref{thm:RestrictionBound});
\item $\nu_{\vb}$ has Fourier decay of level
$
\theta \ll 1/\sqrt{w}
$ (Lemma \ref{eq:bounded});
\item  $\nu_{\vb}$ saves $1/5$ on the $K$-trivial solutions (Cor.~\ref{cor:8.2}).
\end{itemize}

Incorporating these facts into Proposition \ref{BourbakiProp}, we deduce that
$$
\sum_n 1_{A_{\vb}}(n) \nu_{\vb}(n) \ll_{\vc,k,\eps} \frac{N_{\vb}}{\min\set{\log\log w,\ \log N_{\vb}}^{s-2-\eps}}.
$$
Recalling \eqref{Ww}, \eqref{Wbound} and \eqref{NDefn}, we finally deduce that
$$
\delta \ll_{\vc, k, \eps} (\log\log \log X)^{-(\frac{s}{2} - 1 - \eps)}.
$$
This completes the proof of Theorem \ref{thm:main}.

\end{document}